\documentclass{aims}
\usepackage{amsmath}
\usepackage{color}
\usepackage{paralist}
\usepackage{graphics} 
\usepackage{epsfig} 
\usepackage{graphicx}  \usepackage{epstopdf}
\usepackage[colorlinks=true]{hyperref}
\hypersetup{urlcolor=blue, citecolor=red}

\def\currentvolume{X}
 \def\currentissue{X}
  \def\currentyear{200X}
   \def\currentmonth{XX}
    \def\ppages{X--XX}
     \def\DOI{10.3934/xx.xx.xx.xx}

\usepackage[pagewise]{lineno} 

\newtheorem{theorem}{Theorem}[section]

\newtheorem{lemma}[theorem]{Lemma}
\newtheorem{proposition}{Proposition}

\theoremstyle{definition}

\newcommand{\ep}{\varepsilon}

\newcommand{\udm}{u_{d,m}}
\newcommand{\dep}{d_{\varepsilon}}
\newcommand{\mep}{m_{\varepsilon}}
\newcommand{\uep}{u_\varepsilon}
\newcommand{\vep}{v_\varepsilon}

\newcommand{\R}{\mathbb{R}}

\newcommand{\e}{\varepsilon}
\newcommand{\sfrac}[2]{\text{\small $\dfrac{#1}{#2}$}}

\title[ratio of total masses of species to resources] 
{On the ratio of total masses of species to resources for a logistic equation with Dirichlet boundary condition}

\author[Jumpei Inoue]{}


\email{j-inoue@toki.waseda.jp}

\thanks{The author is supported by JSPS KAKENHI Grant-in-Aid Grant Number 21J14292}

\begin{document}
\maketitle

\centerline{\scshape Jumpei Inoue }
\medskip
{\footnotesize
 \centerline{Department of Pure and Applied Mathematics}
 \centerline{Graduate School of Fundamental Science and Engineering}
   \centerline{Waseda University}
   \centerline{3-4-1 Ohkubo, Shinjuku-ku, Tokyo, 164-8555, Japan}
} 

\bigskip

\centerline{(Communicated by )}

\begin{abstract}

We consider the stationary problem for a diffusive logistic equation with the homogeneous Dirichlet boundary condition.
Concerning the corresponding Neumann problem, Wei-Ming Ni proposed a question as follows: Maximizing the ratio of the total masses of species to resources.
For this question, Bai, He and Li \cite{BaiHeLi} showed that the supremum of the ratio is 3 in the one dimensional case, and the author and Kuto \cite{InoueKuto} showed that the supremum is infinity in the multi-dimensional ball.
In this paper, we show the same results still hold true for the Dirichlet problem.
Our proof is based on the sub-super solution method and needs more delicate calculation because of the range of the diffusion rate for the existence of the solution.

\end{abstract}

\section{Introduction}

This paper is concerned with the following stationary problem for a diffusive logistic equation with the homogeneous Dirichlet boundary condition:
\begin{equation}\label{main_eq}
  \begin{cases}
    d\,\Delta u + u(m(x)-u)=0,\;\; u>0 &\text{in}\;\;\Omega, \\
    \; u = 0 &\text{on}\;\;\partial\Omega,
  \end{cases}
\end{equation}
where $\Omega \subseteq \R^n$ is a bounded domain with a smooth boundary; $d$ is a positive constant; $m(x)$ is a measurable function belonging to
\[ L_+^\infty (\Omega) := \{\, f \in L^\infty(\Omega) \mid f(x) \ge 0 \;\text{a.e.}\; x \in \Omega,\; \|f\|_{L^\infty} > 0 \,\}. \]
The unknown function $u(x)$ represents the distribution of the species and $m(x)$ can be interpreted as the distribution of resources, and moreover, $d$ represents the diffusion rate of the species.
The boundary condition assumes that the habitat region $\Omega$ is surrounded by an inhospitable environment for the species.
The existence, uniqueness, and stability of positive solutions were obtained by Cantrell and Cosner \cite{CantrellCosner}.
Before we introduce the existence and uniqueness result by \cite{CantrellCosner}, we prepare an eigenvalue problem corresponding to \eqref{main_eq}:
\begin{equation}\label{EVP}
  \begin{cases}
    \Delta \phi + \lambda m(x) \phi = 0 &\text{in}\;\;\Omega, \\
    \; \phi = 0 &\text{on}\;\;\partial\Omega.
  \end{cases}
\end{equation}
It is known that \eqref{EVP} has a positive sequence of eigenvalues depending on $m \in L_+^\infty(\Omega)$;
\[ (0 <\,)\, \lambda_1(m) < \lambda_2(m) \le \lambda_3(m) \le \dots \quad\text{with}\quad \lim_{j \to \infty} \lambda_j(m)=+\infty, \]
see \cite{CantrellCosner} and the references therein.

\begin{proposition}(\cite[Theorem 2.1]{CantrellCosner})\label{exist-sol}
  For all $m \in L_+^\infty(\Omega)$ and $d \in (0,1/\lambda_1(m))$, \eqref{main_eq} has a unique solution $\udm(x)$ in the class of $W^{2,p}(\Omega) \cap W_0^{1,p}(\Omega)$ for any $p > 1$. Furthermore, $\udm(x)$ is globally asymptotically stable in the sense that it attracts all positive solutions of the corresponding parabolic problem as $t \to \infty$.
\end{proposition}

Related to the diffusive logistic equation \eqref{main_eq}, we deal with the following question:
``What is the supremum of
\begin{equation}\label{ratio}
  \frac{\|\udm\|_{L^1(\Omega)}}{\|m\|_{L^1(\Omega)}} = \frac{\int_\Omega \udm \,dx}{\int_\Omega m \,dx}
\end{equation}
for any $(d,m) \in (0,1/\lambda_1(m)) \times L_+^\infty(\Omega)$?''
The above question was proposed by Wei-Ming Ni in the setting of the Neumann boundary condition.
Replacing the Dirichlet boundary condition with the Neumann one, it is also well-known that there exists a unique solution of \eqref{main_eq} for any $(d,m) \in (0,\infty) \times L_+^\infty(\Omega)$, see \cite{CantrellCosner}.
For the Neumann boundary condition, the supremum of \eqref{ratio} is known.
In the one-dimensional case, Bai, He and Li \cite{BaiHeLi} proved that the supremum is $3$.
On the other hand, in the case of the higher dimensions ($n \ge 2$), the author and Kuto \cite{InoueKuto} showed that the supremum is infinity when the domain is a multi-dimensional ball.
Heo and Kim \cite{HeoKim} improved the higher-dimensional result with respect to the setting of the general domain with a smooth boundary.

In this paper, we consider the Ni's problem for the Dirichlet boundary condition, and we obtain the same supremum of \eqref{ratio} as the Neumann boundary condition for any dimension as follows:

\begin{theorem}\label{main_thm1}
  Assume that $n=1$, and $\Omega = (-1,1)$.
  Let $\udm(x)$ be a solution of \eqref{main_eq}.
  For any $(d,m) \in (0,1/\lambda_1(m)) \times L_+^\infty(\Omega)$, the following inequality holds:
  \begin{equation}\label{result1}
    \frac{ \|\udm\|_{L^1(\Omega)} }{ \|m\|_{L^1(\Omega)} } < 3.
  \end{equation}
  Moreover,
  \begin{equation}\label{result2}
    \sup_{(d,m) \in (0,1/\lambda_1(m)) \times L_+^\infty(\Omega)} \frac{\|\udm\|_{L^1(\Omega)}}{\|m\|_{L^1(\Omega)}} = 3.
  \end{equation}
\end{theorem}

\begin{theorem}\label{main_thm2}
  Assume that $n \ge 2$ and $\Omega = B_0(1)$.
  Let $\udm(x)$ be a solution of \eqref{main_eq}, then it holds that
  \begin{equation}\label{result3}
    \sup_{(d,m) \in (0,1/\lambda_1(m)) \times L_+^\infty(B_0(1))} \frac{\|\udm\|_{L^1(B_0(1))}}{\|m\|_{L^1(B_0(1))}} = +\infty.
  \end{equation}
\end{theorem}
We use the notation $B_0(r) := \{ x \in \R^n \mid |x| < r \}$.

It is noted that, for any fixed $m\in L^{\infty}_+(\Omega)$, the range of $d$ for the existence of positive solutions depends on the boundary conditions.
In the Dirichlet boundary condition, the range is $(0, 1/\lambda_1(m))$,
while in the Neumann boundary condition, the range is $(0,\infty)$.
From the view point of the bifurcation diagram,
the projection of the branch of positive solutions on $d$ axis is $(0,1/\lambda_1(m))$ in the Dirichlet boundary condition,
and the branch bifurcates from the trivial solution at $d = 1/\lambda_1$.
Moreover, it is well known that, for each $d \in (0,1/\lambda_{1}(m))$,
the positive solution for the Dirichlet problem is less than the one for the Neumann problem.
Obviously, it holds that
\[ \inf_{(d,m) \in (0,1/\lambda_1(m)) \times L^{\infty}_+(\Omega)} \frac{\|\udm\|_{L^1}}{\|m\|_{L^1}} = 0 \]
for the Dirichlet problem.
Here, we note that the infimum of the ratio over $(d,m) \in (0,\infty) \times L_+^\infty(\Omega)$ attains $1$ in the case where $m(x)$ is constant for the Neumann problem.
On the other hand, the above results show that the supremum of the ratio is the same regardless of the boundary conditions.

The proof is based on the sub-super solution method.
The key step is to construct the family of sub-solutions $\{ \underline{u}_\e = \underline{u}_{\dep, \mep} \}_{\e>0}$ which satisfies
\[
  \lim_{\e \to 0} \frac{ \|\underline{u}_\e\|_{L^1} }{ \|\mep\|_{L^1} } = 
  \begin{cases} 3 & (n=1) \\ + \infty & (n \ge 2) \end{cases}
\]
with some suitable choice $(\dep, \mep) \in (0, 1/\lambda_1(\mep)) \times L_+^\infty(\Omega)$.
Different from the Neumann problem, the construction requires a somewhat delicate treatment such that the range of $\dep$ is restricted to $(0, 1/\lambda_1(\mep))$.


We refer to \cite{DeAngelisZhangNiWang, Inoue, InoueKuto, LiLou, LiangLou, Lou1, Mazzari, MazzariNadinPrivat, NagaharaYanagida} for the Neumann boundary condition, and \cite{HeLamLouNi, HeNi1, HeNi2, HeNi3, HeNi4, HeNi5, Lou3, LouWang} for applications to a class of diffusive Lotka-Volterra systems.
See also the book chapters \cite{LamLou}, \cite{Lou2} and \cite{Ni} to know recent studies for \eqref{main_eq} and related problems.

This paper consists of four sections.
Section 2 is devoted to the proof of the one-dimensional result, that is, Theorem \ref{main_thm1}.
Theorem \ref{main_thm2} is proved in Section 3.
Finally, the appendix shows an estimate for the least eigenvalue in two dimension.

\section{Proof of 1-dimensional results}

Assume $n=1$ and let $\Omega$ be an open interval $(-1,1)$.
In this section we show Theorem \ref{main_thm1}.

\begin{lemma}\label{lem1-1}
  Let $\udm$ be the solution of \eqref{main_eq}.
  Then the following inequality holds for any $(d,m) \in (0,1/\lambda_1(m)) \times L_+^\infty(-1,1)$:
  \[ \frac{ \|\udm\|_{L^1(-1,1)} }{ \|m\|_{L^1(-1,1)} } < 3. \]
\end{lemma}
\begin{proof}
  Let $\overline{u}_{d,m}$ be a solution of the Neumann problem:
  \begin{equation}\label{Neumann-1dim}
    \begin{cases}
      d\,u'' + u(m(x)-u)=0,\;\; u>0 &(-1<x<1), \\
      \; u'(-1) = u'(1) = 0.
    \end{cases}
  \end{equation}
  The existence and uniqueness of $\overline{u}_{d,m}$ can be found in \cite{CantrellCosner} and the references therein.
  We can see that $\overline{u}_{d,m}$ is a super-solution of \eqref{main_eq} and $\udm$ itself is a sub-solution.
  Hence $\udm \le \overline{u}_{d,m}$ for $x \in (-1,1)$.
  The definition of the sub-super solution can be found in \cite{Du, InoueKuto}.
  Referring to the results for the Neumann boundary condition \cite[Theorem 1.1]{BaiHeLi}, we obtain
  \[ \frac{\int_{-1}^1 \udm \,dx}{\int_{-1}^1 m \,dx} \le \frac{\int_{-1}^1 \overline{u}_{d,m} \,dx}{\int_{-1}^1 m \,dx} < 3 \]
  for any $(d,m) \in (0,1/\lambda_1(m)) \times L_+^\infty(-1,1)$.
\end{proof}
Lemma \ref{lem1-1} proves \eqref{result1} in Theorem \ref{main_thm1}.

Next we show \eqref{result2}.
For small $\e \in (0,1)$, we set
\begin{equation}\label{resource1}
  m(x) = \mep(x) =
  \begin{cases}
    1 / \e &\text{for}\; x \in (-\e,\e), \\
    0 &\text{for}\; x \in (-1,-\e) \cup (\e,1).
  \end{cases} 
\end{equation}
It follows that $\|\mep\|_{L^1(-1,1)} = 2$.
To derive a range of the diffusion coefficient for the existence of the solution $\udm$, we consider the following eigenvalue problem:
\begin{equation}\label{EVP-1dim}
  \begin{cases}
    \phi'' + \lambda \mep(x) \phi = 0 &(-1<x<1), \\
    \; \phi(-1) = \phi(1) = 0.
  \end{cases}
\end{equation}
If we find the least eigenpair $( \lambda_1(\mep), \phi_1 )$ of
\begin{equation}\label{EVP-1dim-sub}
  \begin{cases}
    \phi'' + \lambda \mep(x) \phi = 0 &(0<x<1), \\
    \; \phi'(0) = \phi(1) = 0,
  \end{cases}
\end{equation}
then $\lambda_1(\mep)$ is also the least eigenvalue of \eqref{EVP-1dim} and $\phi_1(x) \;(0<x\le1); \; \phi_1(-x) \;(-1<x<0)$ is the eigenfunction of \eqref{EVP-1dim} corresponding to $\lambda_1(\mep)$.

For $x \in (\e,1)$, by the boundary condition $\phi(1) = 0$, one can see that
\[ \phi(x) = A(x-1) \quad (\e < x < 1), \]
where $A$ is an arbitrary constant.
For $x \in (0,\e)$, using $\phi'(0) = 0$, we have
\[ \phi(x) = B \cos \sqrt{\frac{\lambda}{\e}}x \quad (0 < x < \e), \]
where $B$ is also an arbitrary constant.
Connecting these two functions at $x=\e$ for $C^1$-class, we have
\[ A(\e-1) = B \cos \sqrt{\lambda \e} \quad\text{and}\quad A = -B \sqrt{\frac{\lambda}{\e}} \sin \sqrt{\lambda \e}. \]
Hence, the eigenvalue problem \eqref{EVP-1dim-sub} has non-trivial solutions $\phi(x) \not\equiv 0$ if and only if
\begin{equation}\label{coeff-mat}
  \begin{bmatrix} \e-1 & -\cos \sqrt{\lambda \e} \\ 1 & \sqrt{\frac{\lambda}{\e}} \sin \sqrt{\lambda \e} \end{bmatrix}
  \begin{bmatrix} A \\ B \end{bmatrix}
  =
  \begin{bmatrix} 0 \\ 0 \end{bmatrix}
\end{equation}
have non-trivial solutions $(A,B) \neq (0,0)$.
That is, the determinant of the coefficient matrix of \eqref{coeff-mat} is zero:
\begin{equation}\label{det0}
  \tan \sqrt{\lambda \e} = \frac{1}{1-\e} \sqrt{ \frac{\e}{\lambda} }.
\end{equation}
Intersections of two graphs $f(\lambda) := \tan \sqrt{\lambda \e} \;\; (\lambda > 0)$ and $g(\lambda) := \sqrt{\e} / [(1-\e)\sqrt{\lambda}] \;\; (\lambda > 0)$ are the eigenvalues of \eqref{EVP-1dim-sub}.
Therefore, we obtain estimates of each eigenvalue:
\[ (k-1)^2 \frac{\pi^2}{\e} < \lambda_k(\mep) < \left( k-\frac{1}{2} \right)^2 \frac{\pi^2}{\e} \quad (k =1,2,\dots) . \]
More detailed estimates of the least eigenvalue $\lambda_1(\mep)$ are needed.

\begin{lemma}\label{LEV1dim}
  The least eigenvalue is estimated as follows:
  \begin{equation}\label{principalEV1}
    \frac{1}{\lambda_1(\mep;1)} \ge 1-\e,
  \end{equation}
  and
  \begin{equation}\label{limit-ev1}
    \lim_{\e \to 0} \lambda_1(\mep;1) = 1.
  \end{equation}
\end{lemma}

In what follows, we use the notation that
\[ \lambda_1 = \lambda_1(m) = \lambda_1(m;n), \]
where $n$ is the dimension number.

\begin{proof}
  We know an inequality $\theta \le \tan \theta$ for $\theta \in (0, \pi/2)$.
  Hence, substituting $\theta = \sqrt{\lambda_1 \e}$ and employing \eqref{det0}, we obtain \eqref{principalEV1}.
  Applying the changing variables $\theta = \sqrt{\lambda \e}$ to \eqref{det0} yields
  \begin{equation}\label{least_ev1}
    \theta \tan \theta = \frac{\e}{1-\e} \quad (0 < \theta < \pi/2).
  \end{equation}
  Thanks to the Maclaurin series:
  \begin{align*}
    & \theta \tan \theta = \theta^2 + \frac{1}{3} \theta^4 + o(\theta^5) \quad (\theta \to 0), \\
    & \frac{\e}{1-\e} = \e + \e^2 + o(\e^2) \quad (\e \to 0),
  \end{align*}
  we have the following inequality for efficiently small $\e$:
  \[ \sqrt{\e} \tan \sqrt{\e} \le \theta_1^\e \tan \theta_1^\e = \frac{\e}{1-\e}, \]
  where $\theta_1^\e \in (0,\pi/2)$ is the unique solution of \eqref{least_ev1}.
  We obtain $\sqrt{\e} \le \theta_1^\e$ for small $\e$.
  On the other hand, we have inverse inequality
  \[ \sqrt{2\e} \tan \sqrt{2\e} \ge \theta_1^\e \tan \theta_1^\e = \frac{\e}{1-\e}, \]
  and so $\theta_1^\e \le \sqrt{2\e}$ for efficiently small $\e$.
  Consequently, $\sqrt{\e} \le \theta_1^\e = \sqrt{\lambda_1 \e} \le \sqrt{2\e}$ holds for small $\e$, and we derive
  \[ \frac{1}{2} \le \lambda_1(\e) \le 1 \quad \text{for efficiently small}\; \e. \]
  By the above inequality and \eqref{principalEV1}, the desired estimate \eqref{limit-ev1} is proved.
\end{proof}
Now we prove \eqref{result2} in Theorem 1.1.
Let the resource function $m(x)$ be \eqref{resource1}, then thanks to \eqref{principalEV1}, we can set the diffusive coefficient $d = \dep = \sqrt{\e} \in (0,1-\e)$.
If we find the solution $U_\ep(x)$ of
\begin{equation}\label{dirichlet-1dim-half}
  \begin{cases}
    \sqrt{\e}\, U'' + U(\mep(x)-U) = 0,\;\; U>0 &(0<x<1), \\
    \; U'(0) = U(1) = 0,
  \end{cases}
\end{equation}
then $\uep(x) := U_\e(x) \;(0 \le x \le 1); \;\; U_\ep(-x) \;(-1 \le x < 0)$ is the solution of \eqref{main_eq}.
Because of
\[ \frac{\int_{-1}^1 \uep \,dx}{\int_{-1}^1 \mep \,dx} = \frac{2\int_0^1 U_\ep(x) \,dx}{2} = \int_0^1 U_\e(x) \,dx, \]
Hence we only show $\lim_{\e \to 0} \int_0^1 U_\e(x) \,dx = 3$ to prove \eqref{result2}.
Let $\vep$ be the solution of the Neumann problem
\begin{equation}\label{neumann-1dim}
  \begin{cases}
    \sqrt{\e}\, v'' + v(\mep(x)-v) = 0,\;\; U>0 &(0<x<1), \\
    \; v'(0) = v'(1) = 0.
  \end{cases}
\end{equation}
This is introduced in \cite{BaiHeLi} as the maximizing sequence, that is, $\lim_{\e \to 0} \int_0^1 \vep(x) \,dx = 3$.
Furthermore, detailed estimates about $\vep$ are derived in \cite{Inoue}.
We define two functions $\overline{U}_\e$ and $\underline{U}_\e$ as follows:
\[ \overline{U}_\e(x) := \vep(x) \quad (0 \le x \le 1) \]
and
\[ \underline{U}_\e(x) := \left( 1-\e^{1/4} \right) \left( \vep(x) - \vep(1) \right) \quad (0 \le x \le 1). \]
It is easy to see that $\overline{U}_\e$ is a super-solution of \eqref{dirichlet-1dim-half}.

\begin{lemma}\label{subsol-1dim}
  For sufficiently small $\e \in (0,1)$, the function $\underline{U}_\e$ is a sub-solution of \eqref{dirichlet-1dim-half}.
\end{lemma}
\begin{proof}
  First of all, the boundary condition is satisfied:
  \[ \underline{U}'_\e(0) = \left( 1-\e^{1/4} \right) v'_\e(0) = 0 \quad \text{and} \quad \underline{U}_\e(1) = 0. \]
  For any $x \in (0,1)$, we have
  \begin{align*}
    & \sqrt{\e}\, \underline{U}''_\e(x) + \underline{U}_\e(x) \left( \mep(x) - \underline{U}_\e(x) \right) \\
    &= \left( 1-\e^{1/4} \right) \left[ \sqrt{\e} v''_\e(x) + (\vep(x)-\vep(1)) \left( \mep(x) - \left( 1-\e^{1/4} \right) (\vep(x)-\vep(1)) \right) \right].
  \end{align*}
  Substituting $\sqrt{\e} v''_\e = - \vep(\mep - \vep)$, we obtain
  \begin{align*}
    & \sqrt{\e}\, \underline{U}''_\e + \underline{U}_\e \left( \mep - \underline{U}_\e \right) \\
    &= \left( 1-\e^{1/4} \right) \left[ \e^{1/4}\vep^2 + \vep(1) \left( 2(1-\e^{1/4}) \vep - \mep - (1-\e^{1/4}) \vep(1) \right) \right].
  \end{align*}
  From \cite{Inoue}, we use some estimates of $\vep(x)$ as follows:
  \begin{align}
    & \lim_{\e \to 0} \frac{\vep(1)}{\sqrt{\e}} = C_1 \;(\text{some positive constant}), \label{I1} \\
    & \lim_{\e \to 0} \sqrt{\e}\,\vep(x) = \frac{3}{2} \;\;\text{for any}\; x \in [0,\e]. \label{I2}
  \end{align}
  We note that $\vep(x)$ is a monotonically decreasing function with respect to $x$.
  Due to $\mep(x) = 1/\e$ for $x \in (0,\e)$, if $\e$ is sufficiently small, then we have
  \begin{align*}
    & \sqrt{\e}\, \underline{U}''_\e + \underline{U}_\e \left( \mep - \underline{U}_\e \right) \\
    &= \left( 1-\e^{1/4} \right) \left[ \e^{1/4}\vep^2 + \vep(1) \left( 2(1-\e^{1/4}) \vep - \frac{1}{\e} - (1-\e^{1/4}) \vep(1) \right) \right] \\
    &\ge \left( 1-\e^{1/4} \right) \left( \e^{1/4}\vep^2 - \frac{\vep(1)}{\e} \right) > 0
  \end{align*}
  for $x \in (0,\e)$.
  Here, we used \eqref{I1}, \eqref{I2}, and the monotonicity of $\vep(x)$.
  Next, for $x \in (\e,1)$, one can see that
  \begin{align*}
    & \sqrt{\e}\, \underline{U}''_\e + \underline{U}_\e \left( \mep - \underline{U}_\e \right) \\
    &= \left( 1-\e^{1/4} \right) \left[ \e^{1/4}\vep^2 + \vep(1) \left( 2(1-\e^{1/4}) \vep - (1-\e^{1/4}) \vep(1) \right) \right] \\
    &\ge \left( 1-\e^{1/4} \right) \left[ \e^{1/4}\vep^2 + (1-\e^{1/4}) \vep(1)^2 \right] > 0,
  \end{align*}
  where we again used the monotonicity of $\vep(x)$.
  Summarizing, we obtain
  \[ \sqrt{\e}\, \underline{U}''_\e + \underline{U}_\e \left( \mep - \underline{U}_\e \right) > 0 \quad (0<x<1) \]
  for sufficiently small $\e$.
\end{proof}

Therefore, we have the inequality
\[ \underline{U}_\e(x) \le U_\e(x) \le \overline{U}_\e(x) \quad (0 \le x \le 1), \]
and so, integrating each side, we have
\[ \left( 1-\e^{1/4} \right) \left( \int_0^1 \vep(x) \,dx - \vep(1) \right) \le \int_0^1 U_\e(x) \,dx \le \int_0^1 \vep(x) \,dx. \]
Remembering $\lim_{\e \to 0} \int_0^1 \vep \,dx = 3$ and using again \eqref{I1}, we obtain $\lim_{\e \to 0} \int_0^1 U_\e \,dx = 3$.
Finally, we prove Theorem \ref{main_thm1}.

\section{Proof of higher-dimensional results}

Hereafter we consider \eqref{main_eq} in the case where $\Omega$ is the multi-dimensional unit ball $B_0(1)$ with $n \ge 2$.
For any $\e \in (0,1)$, we set the resource function as follows:
\begin{equation}\label{resource2}
  m(x) = \mep(x) =
  \begin{cases}
    1 / \e^n &\text{for}\; x \in \overline{B_0(\e)}, \\
    0 &\text{for}\; x \in \overline{B_0(1)} \setminus \overline{B_0(\e)}.
  \end{cases} 
\end{equation}
It follows that $\|\mep\|_{L^1(B_0(1))}=|B_0(1)|$, where $|B_0(1)|$ denotes the volume of $B_0(1)$.

First of all, we derive an estimate about the least eigenvalue $\lambda_1(\mep)$ for \eqref{EVP}.
The characterization of the least eigenvalue for \eqref{EVP} is well-known, that is,
\[ \lambda_1(m) = \inf_{0 \neq \phi \in H_0^1(\Omega)} \frac{\int_{\Omega} |\nabla \phi|^2 \,dx}{\int_{\Omega} m\phi^2 \,dx}. \]
We substitute 
\[
  \phi(x) =
  \begin{cases}
    -|x|+\e &\text{for}\; x \in \overline{B_0(\e)}, \\
    0 &\text{for}\; x \in \overline{B_0(1)} \setminus \overline{B_0(\e)}.
  \end{cases}
\]
as a test function in $H_0^1 (\Omega)$.
Then straightforward calculations yield
\[ \lambda_1(\mep) \le \e^n \frac{ \int_0^1 \phi_r^2 r^{n-1} \,dr }{ \int_0^\e \phi^2 r^{n-1} \,dr } = \e^n \frac{ \int_0^\e (-1)^2 r^{n-1} \,dr }{ \int_0^\e (-r+\e)^2 r^{n-1} \,dr } = \frac{(n+1)(n+2)}{2} \e^{n-2}, \]
and so we have the following estimate
\begin{equation}\label{principalEV2}
  \frac{1}{\lambda_1(\mep;n)} \ge \frac{2}{(n+1)(n+2)\e^{n-2}}.
\end{equation}

\begin{theorem}\label{subsol-2dim}
  Assume that the dimension number $n$ satisfies $n \ge 2$ and the domain satisfies $\Omega = B_0(1)$.
  Let the resource function $m(x)$ be defined in \eqref{resource2}.
  Then there exist positive constants $c_1$ and $c_2$ depending only on $n$ such that the solution $\uep(x)$ of
  \begin{equation}\label{highdim-eq}
    \begin{cases}
      \frac{c_1}{\e^{n-2}} \Delta u + u (\mep(x)-u) = 0, \;\; u>0 &\text{in}\;\; B_0(1),\\
      \; u = 0 &\text{on}\;\; \partial B_0(1)
    \end{cases}
  \end{equation}
  satisfies
  \[ \frac{ \|\uep\|_{L^1(B_0(1))} }{ \|\mep\|_{L^1(B_0(1))} } \ge c_2 \left( \frac{n}{e} |\log \e| + 1 - \frac{2}{e} \right) \]
  for any sufficiently small $\e \in (0,1)$.
\end{theorem}

By setting $\e \to 0$, Theorem \ref{subsol-2dim} immediately leads to Theorem \ref{main_thm2}.

\begin{proof}
We set the diffusive coefficient as
\[
  d = \dep := \frac{c_1}{\e^{n-2}},
\]
where $c_1$ is a positive constant independent of $\e$ and will be defined later.
Employing \eqref{principalEV2}, we impose
\begin{equation}\label{C1}
  c_1 < \frac{2}{(n+1)(n+2)} \tag{C1}
\end{equation}
to ensure the existence of the solution $\uep$ for \eqref{highdim-eq}.

We introduce two functions $\overline{u}_\e(x)$ and $\underline{u}_\e(x)$ defined over $\overline{B_0(1)}$ as 
\begin{equation}\label{super_sol}
  \overline{u}_\e(x) := \frac{1}{\e^{n}} \quad \text{for}\;\; x \in \overline{B_0(1)}
\end{equation}
and
\begin{equation}\label{sub_sol}
  \underline{u}_\e(x) :=
  \begin{cases}
    \; \dfrac{c_2}{\e^n} e^{-|x|^n / \e^n} - \dfrac{c_2}{e} &\text{for}\;\; x \in \overline{B_0(\e)}, \vspace{1mm}\\
    \; \dfrac{c_2}{e|x|^n} - \dfrac{c_2}{e} &\text{for}\;\; x \in \overline{B_0(1)} \setminus \overline{B_0(\e)},
  \end{cases}
\end{equation}
where a positive constant $c_2$ will be determined later independently of $\e$.
We will verify that \eqref{super_sol} is a super-solution for \eqref{highdim-eq} and \eqref{sub_sol} is a sub-solution, respectively.
The super-solution \eqref{super_sol} is introduced in \cite{InoueKuto}.
The constant $1/\e^n$ is derived from the maximum value of $\mep(x)$.
In \cite{InoueKuto}, a sub-solution for the Neumann boundary condition is introduced as
\[
  V_\e(x) :=
  \begin{cases}
    \; \dfrac{c_2}{\e^n} e^{-|x|^n / \e^n} &\text{for}\;\; x \in \overline{B_0(\e)}, \vspace{1mm}\\
    \; \dfrac{c_2}{e|x|^n} &\text{for}\;\; x \in \overline{B_0(1)} \setminus \overline{B_0(\e)}.
  \end{cases}
\]
We modified $V_\e(x)$ to \eqref{sub_sol} for satisfying the Dirichlet boundary condition.

The following procedure is the same as the proof of \cite[Theorem 2.2]{InoueKuto}.
First of all, it is easily verified that $\overline{u}_\e(x)$ satisfies
\[
  \frac{c_1}{\e^{n-2}} \Delta \overline{u}_\e + \overline{u}_\e (\mep(x) - \overline{u}_\e) =
  \begin{cases}
    \; 0 &\text{for}\;\; x \in B_0(1), \\
    - \overline{u}_\e^{2} < 0 &\text{for}\;\; x \in B_0(1) \setminus \overline{B_0(\e)},
  \end{cases}
\]
and $\overline{u}_\e > 0 \;\text{on}\; \partial B_0(1)$.
The above inequalities show that $\overline{u}_\e(x)$ is a super-solution for \eqref{highdim-eq}.

In the next, we find the range of parameters $c_1,\, c_2$, so that $\underline{u}_\e(x)$ will be a sub-solution for \eqref{highdim-eq}.
Thanks to the fact that $\underline{u}_\e \in C^2(\overline{B_0(1)} \setminus \{ |x|=\e \}) \cap C^1(\overline{B_0(1)})$, it suffices to show that
\begin{equation}\label{sub_sol_diff}
  \begin{cases}
    \frac{c_1}{\e^{n-2}} \left( \underline{u}_\e'' + \sfrac{n-1}{r} \underline{u}_\e' \right) + \underline{u}_\e (\mep(r) - \underline{u}_\e) \ge 0 &\text{for}\;\; r \in (0,\e) \cup (\e,1),\\
    \; \underline{u}_\e'(0) = 0,\quad \underline{u}_\e(1) \le 0. &
  \end{cases}
\end{equation}
Since $\underline{u}_\e(x)$ and $\mep(x)$ are radial functions, it is not confusing that we use the same notation $\underline{u}_\e(r) = \underline{u}_\e(x)$, $\mep(r) = \mep(x)$ for $r = |x| \in [0,1]$, that is,
\[
  \underline{u}_\e(r) =
  \begin{cases}
    \; \dfrac{c_2}{\e^n} e^{-r^n / \e^n} - \dfrac{c_2}{e} &(0 \leq r \leq \e), \vspace{1mm} \\
    \; \dfrac{c_2}{e r^n} - \dfrac{c_2}{e} &(\e < r \leq 1),
  \end{cases} \qquad
  \mep(r) =
  \begin{cases}
    1/\e^n & (0 \leq r \leq \e), \\
    \; 0 & (\e < r \leq 1).
  \end{cases} 
\]
Note also that the prime symbol $'$ represents the derivative by $r$.
We prepare the derivatives of $\underline{u}_\e(r)$:
\[
  \underline{u}_\e'(r) =
    \begin{cases}
      \; - \dfrac{c_2 n r^{n-1}}{\e^{2n}} e^{-r^n / \e^n} &(0 \leq r \leq \e), \vspace{1mm}\\
      \; - \dfrac{c_2 n}{er^{n+1}} &(\e < r \leq 1)
    \end{cases}
\]
and
\[
  \underline{u}_\e''(r) =
    \begin{cases}
      \; \dfrac{c_2 n(n-1) r^{n-2}}{\e^{2n}} \biggl( \dfrac{n r^n}{(n-1)\e^n}-1 \biggr) e^{-r^n / \e^n} &(0 \leq r \leq \e), \vspace{1mm}\\
      \; \dfrac{c_2 n(n+1)}{e r^{n+2}} &(\e < r \leq 1).
    \end{cases}
\]
It can be easily seen that $\underline{u}_\e'(0) = 0$ and $\underline{u}_\e(1) = 0$.
Now we suppose that
\begin{equation}\label{C2}
  1-2n(n-1)c_1 > ec_2. \tag{C2}
\end{equation}
The reason why we imposed \eqref{C2} will be seen in the following calculation.
For $r \in (0,\e)$, we have
\begin{align*}
  & \frac{c_1}{\e^{n-2}} \left( \underline{u}_\e'' + \frac{n-1}{r} \underline{u}_\e' \right) + \underline{u}_\e \left( \frac{1}{\e^n} - \underline{u}_\e \right) \\
  &= \left( \frac{c_1c_2n^2}{\e^{4n-2}}r^{2n-2} - \frac{2c_1c_2n(n-1)}{\e^{3n-2}}r^{n-2} + \frac{c_2^2}{\e^{2n}} + \frac{2c_2^2}{e \e^n} \right) e^{-r^n/\e^n} \\
  &\quad \quad - \frac{c_2^2}{\e^{2n}} e^{-2r^n/\e^n} - \frac{c_2^2}{e^2} - \frac{c_2}{e \e^n} \\
  &\ge \left( \frac{c_1c_2n^2}{\e^{4n-2}} \cdot 0 - \frac{2c_1c_2n(n-1)}{\e^{3n-2}} \e^{n-2} + \frac{c_2}{\e^{2n}} + \frac{2c_2^2}{e \e^n} \right) e^{-r^n/\e^n} - \frac{c_2^2}{\e^{2n}} \cdot 1 - \frac{c_2^2}{e^2} - \frac{c_2}{e \e^n} \\
  &= c_2 \left[ \left( \frac{1-2c_1n(n-1)}{\e^{2n}} + \frac{2c_2}{e \e^n} \right) e^{-r^n/\e^n} - \frac{c_2}{\e^{2n}} - \frac{1}{e \e^n} - \frac{c_2}{e^2} \right].
\end{align*}
Thanks to \eqref{C2}, the coefficient of $e^{-r^n/\e^n}$ is positive.
Continuing, we can see that
\begin{align*}
  &\ge c_2 \left[ \left( \frac{1-2c_1n(n-1)}{\e^{2n}} + \frac{2c_2}{e \e^n} \right) e^{-1} - \frac{c_2}{\e^{2n}} - \frac{1}{e \e^n} - \frac{c_2}{e^2} \right] \\
  &= c_2 \left[ \left( \frac{1-2c_1n(n-1)}{e}-c_2 \right) \frac{1}{\e^{2n}} + \left( \frac{2c_2}{e^2} - \frac{1}{e} \right) \frac{1}{\e^n} - \frac{c_2}{e^2} \right].
\end{align*}
By virtue of \eqref{C2} again, the coefficient of the main term $O(1/\e^{2n})$ is positive.
Thus, the differential inequality \eqref{sub_sol_diff} holds in the interval $(0,\e)$ if $\e \in (0,1)$ is sufficiently small.
To prove \eqref{sub_sol_diff} in the interval $(\e,1)$, we impose a condition on the parameters $c_1, c_2$, that is,
\begin{equation}\label{C3}
  2nec_1-c_2 \ge 0. \tag{C3}
\end{equation}
For $r \in (\e,1)$, we obtain
\begin{align*}
  \frac{c_1}{\e^{n-2}} \left( \underline{u}_\e'' + \frac{n-1}{r} \underline{u}_\e' \right) - \underline{u}_\e^2 &= \frac{c_2}{e r^{n+2}} \left( \frac{2nc_1}{\e^{n-2}} - \frac{c_2}{er^{n-2}} \right) + \frac{2c_2^2}{e^2 r^n} - \frac{c_2^2}{e^2} \\
  &\ge \frac{c_2}{e r^{n+2}} \left( \frac{2nc_1}{\e^{n-2}} - \frac{c_2}{e \e^{n-2}} \right) + \frac{2c_2^2}{e^2 \cdot 1} - \frac{c_2^2}{e^2} \\
  &= \frac{c_2}{e \e^{n-2} r^{n+2}} \left(2nc_1 - \frac{c_2}{e} \right) + \frac{c_2^2}{e^2} \\
  &\ge 0,
\end{align*}
where we used \eqref{C3} in the last inequality.
Therefore, $\underline{u}_\e(x)$ defined by \eqref{sub_sol} is a sub-solution for \eqref{highdim-eq}.

Finally, we impose the following:
\begin{equation}\label{C4}
  c_2 < 1 \tag{C4},
\end{equation}
because we need the inequality $(0<\,)\, \underline{u}_\e \le \overline{u}_\e$ over $\overline{B_0(1)}$.
Three conditions \eqref{C2}, \eqref{C3}, and \eqref{C4} are the same as in the case of the Neumann boundary condition \cite{InoueKuto}.
The condition \eqref{C1} is new deriving from the existence of the solution $\uep(x)$.
Here, we introduce the same notation as in \cite{InoueKuto}, that is,
\[ T := \{\, (c_1, c_2) \in \R_{>0}^2 \mid (c_1,c_2) \text{\ satisfies \eqref{C2}, \eqref{C3}, and \eqref{C4}} \,\}. \]
This set forms a triangle in the $(c_1,c_2)$ plane whose vertices are
\[ (c_1,c_2) = (0,0),\, \left( \frac{1}{2n(e^2+n-1)}, \frac{e}{e^2+n-1} \right),\, \left( \frac{1}{2n(n-1)}, 0 \right). \]
In addition, by virtue of \eqref{C1}, we define the followings:
\begin{align*}
  & T_2 := T \cap \{ c_1 < 1/6 \} & \text{in the case where}\; n=2, \\
  & T_n := T, &\text{in the case where}\; n\ge3.
\end{align*}
Therefore, if $n=2$, then $T_n$ forms a quadrilateral, else if $n \ge 3$, then $T_n$ forms a triangle.
Consequently, we obtain the following lemma:

\begin{lemma}\label{sub-super-thm}
  Let $(c_1,c_2) \in T_n$ and $\overline{u}_\e(x), \underline{u}_\e(x)$ be defined by \eqref{super_sol},\; \eqref{sub_sol} respectively. Then the solution $\uep(x)$ of \eqref{highdim-eq} is estimated as the following inequality
  \[ \underline{u}_\e(x) \le \uep(x) \leq \overline{u}_\e(x) \quad \text{for any}\; x \in B_0(1), \]
  if $\e$ is sufficiently small.
\end{lemma}

Finally, we can see that
\begin{align*}
  \|\underline{u}_\e\|_{ L^1( B_0(1) ) } &= A_n \int_0^1 \underline{u}_\e (r) r^{n-1} \,dr \\
  &= \frac{A_n}{n} \left( \frac{c_2n}{e} |\log \e| + c_2 - \frac{2c_2}{e} \right),
\end{align*}
where $A_n$ denotes the surface area of $\partial B_0(1)$.
We also know $\|\mep\|_{ L^1( B_0(1) ) } = |B_0(1)| = A_n/n$.
Hence, we deduce that
\[ \frac{ \|\underline{u}_\e\|_{ L^1( B_0(1) ) } }{ \|\mep\|_{ L^1( B_0(1) ) } } = c_2 \left( \frac{n}{e} |\log \e| + 1 - \frac{2}{e} \right). \]
The above equation and Lemma \ref{sub-super-thm} complete the proof of Theorem \ref{subsol-2dim}.
\end{proof}



\section{Appendix}

We obtain the asymptotic behavior of $\lambda_1(\mep; 1)$ in the one-dimensional case by Lemma \ref{LEV1dim}.
In this section we additionally consider the two-dimensional case.
Suppose $n = 2$ and $\Omega = B_0(1) = \{ x \in \R^2 \mid |x| < 1 \}$.
For small $\e \in (0,1)$, let $m(x)$ be defined by \eqref{resource2}.
We deal with the least eigenvalue $\lambda_1(\mep; 2)\, (\,>0)$ for the eigenvalue problem \eqref{EVP}.

\begin{proposition}\label{LEV2dim}
  We have $\lambda_1(\mep; 2) = O(1/|\log \e|)$ as $\e \to 0$.
\end{proposition}

\begin{proof}
  It is well-known that the eigenfunction corresponding to $\lambda_1(\mep; 2)$ is radial because $\mep(x)$ is radial, and so we focus on the following eigenvalue problem:
  \begin{equation}\label{EVP2dim}
    \begin{cases}
      \phi_{rr} + \dfrac{1}{r}\phi_r + \lambda \mep(r) \phi = 0 &(0<r<1), \\
      \; \phi_r(0) = \phi(1) = 0
    \end{cases}
  \end{equation}
  where
  \[ \mep(r) =
    \begin{cases}
      1 / \e^2 &\text{for}\; r \in [0,\e], \\
      0 &\text{for}\; r \in (\e,1].
    \end{cases} 
  \]
  For the interval $(0,\e)$, we consider
  \[
    \begin{cases}
      \phi_{rr} + \dfrac{1}{r}\phi_r + \dfrac{\lambda}{\e^2} \phi = 0 &(0<r<\e), \\
      \; \phi_r(0) = 0.
    \end{cases}
  \]
  According to the theory of Bessel functions, the solution for the above differential equation is as follows:
  \[ \phi(r) = C_1 J_0 \left( \frac{\sqrt{\lambda}}{\e}r \right) \quad (0<r<\e), \]
  where $C_1$ is as arbitrary constant and $J_\nu(z)$ denotes the Bessel function of the first kind of the order $\nu$.
  See \cite{Watson} and references therein about Bessel functions.
  On the other hand, the differential equation in $(\e, 1)$;
  \[
    \begin{cases}
      \phi_{rr} + \dfrac{1}{r}\phi_r = 0 &(\e<r<1), \\
      \; \phi(1) = 0
    \end{cases}
  \]
  yields
  \[ \phi(r) = C_2 \log r \quad (\e < r <1), \]
  where $C_2$ is an arbitrary constant.
  Then we connect these functions at $r = \e$, that is,
  \[ C_1J_0(\sqrt{\lambda}) = \phi(\e) = C_2 \log \e \quad\text{and}\quad -\frac{\sqrt{\lambda}}{\e} C_1 J_1(\sqrt{\lambda}) = \phi_r(\e) = \frac{C_2}{\e}. \]
  These are equivalent to
  \begin{equation}\label{coeff-mat2dim}
    \begin{bmatrix} J_0(\sqrt{\lambda}) & |\log \e| \\ \sqrt{\lambda} J_1(\sqrt{\lambda}) / \e & 1/\e \end{bmatrix}
    \begin{bmatrix} C_1 \\ C_2 \end{bmatrix}
    =
    \begin{bmatrix} 0 \\ 0 \end{bmatrix}
    .
  \end{equation}
  The eigenvalue problem \eqref{EVP2dim} has non-trivial solutions if and only if the determinant of the coefficient matrix for \eqref{coeff-mat2dim} is zero.
  Hence we have
  \begin{equation}\label{det0-2dim}
    J_0(\sqrt{\lambda}) - \sqrt{\lambda} J_1(\sqrt{\lambda}) |\log \e| = 0.
  \end{equation}
  Now we introduce the zeros of Bessel functions.
  Zeros of $J_0(z)$ denote $z_{0,1},\, z_{0,2},\, z_{0,3}, \dots$ and those of $J_1(z)$ denote $0=z_{1,0},\, z_{1,1},\, z_{1,2},\, z_{1,3}, \dots$.
  It is well-known that
  \[ 0 = z_{1,0} < z_{0,1} < z_{1,1} < z_{0,2} < z_{1,2} < \cdots, \]
  and so $\lambda = (z_{i,j})^2$ are not solutions for \eqref{det0-2dim}.
  Thus, dividing \eqref{det0-2dim} by $\sqrt{\lambda} J_1(\sqrt{\lambda})$, we obtain the following equation:
  \[ \left( g(\sqrt{\lambda}) := \,\right)\, \frac{ J_0(\sqrt{\lambda}) }{ \sqrt{\lambda} J_1(\sqrt{\lambda}) } = |\log \e|. \]
  For fixed $\e$, the graph of $g(z) \;(z>0)$ and the constant $|\log \e|$ have infinitely positive intersections: $z_1^*,\, z_2^*,\, z_3^*, \dots$ which are depending on $\e$.
  Hence the least eigenvalue is characterized by $\lambda_1(\mep; 2) = (z_1^*)^2$ and $0=z_{1,0} < z_1^* < z_{0,1}$.
  By the graph of $g(z)$ and $\lim_{\e \to 0} |\log \e| = +\infty$, we deduce that
  \[ \lim_{\e \to 0} \lambda_1(\mep; 2) = \lim_{\e \to 0} \left( z_1^* \right)^2 = 0. \]
  
  Next, we introduce the asymptotic behavior of $J_0(z)$ and $J_1(z)$ as $z \to 0$:
  \begin{align*}
    J_0(z) &= 1 - \dfrac{z^2}{4} + O(z^4) \quad (z \to 0), \vspace{0.5mm} \\
    J_1(z) &= \dfrac{z}{2} - \dfrac{z^3}{16} + O(z^5) \quad (z \to 0).
  \end{align*}
  We compare
  \[ g(\sqrt{\lambda}) = \frac{1}{\lambda} \cdot \frac{ 1-\lambda/4+O(\lambda^2) }{ 1/2-\lambda/16+O(\lambda^2) } \quad (\lambda \to 0) \]
  and $|\log \e|$ when $\e$ is sufficiently small.
  Then
  \[ g\left( \sqrt{ 4/|\log \e| } \right) \le |\log \e| \le g\left( \sqrt{ 1/|\log \e| } \right) \]
  holds for efficiently small $\e \in (0,1)$.
  Thanks to the fact that $g(z)$ is monotone decreasing for $0=z_{1,0} < z < z_{0,1}$, we obtain
  \[ \sqrt{\frac{1}{|\log \e|}} \le z_1^* \le \sqrt{\frac{4}{|\log \e|}}, \quad\text{that is,}\quad \frac{1}{|\log \e|} \le \lambda_1(\mep;2) \le \frac{4}{|\log \e|} \]
  for efficiently small $\e$.
  The above inequality completes the proof.
\end{proof}



\medskip
Received xxxx 20xx; revised xxxx 20xx.
\medskip

\end{document}